\documentclass[a4paper,12pt]{article}
\usepackage{}
\usepackage{mathrsfs}
\usepackage{algorithm, algpseudocode}
\usepackage{amsmath,amssymb,amsthm,latexsym,amsfonts}
\usepackage{pstricks}

\usepackage{graphicx}
\usepackage{color}
\usepackage{ifpdf}

\usepackage{caption}

\begin{document}

\baselineskip 7mm

\newtheorem{lem}{Lemma}[section]
\newtheorem{thm}[lem]{Theorem}
\newtheorem{cor}[lem]{Corollary}
\newtheorem{exa}[lem]{Example}
\newtheorem{con}[lem]{Conjecture}
\newtheorem{rem}[lem]{Remark}
\newtheorem{obs}[lem]{Observation}
\newtheorem{definition}[lem]{Definition}
\newtheorem{prop}[lem]{Proposition}
\theoremstyle{plain}
\newcommand{\D}{\displaystyle}
\newcommand{\DF}[2]{\D\frac{#1}{#2}}

\renewcommand{\figurename}{Fig.}
\captionsetup{labelfont=bf}

\title{ The (vertex-)monochromatic index of a graph\footnote{Supported by NSFC No.11371205
and 11531011, ``973" program No.2013CB834204, and PCSIRT.}
}

\author{\small Xueliang~Li, Di~Wu\\
\small Center for Combinatorics and LPMC \\
\small Nankai University, Tianjin 300071, China\\
\small lxl@nankai.edu.cn; wudiol@mail.nankai.edu.cn}
\date{}
\maketitle

\begin{abstract}
A tree $T$ in an edge-colored graph $H$ is called a \emph{monochromatic tree}
if all the edges of $T$ have the same color.
For $S\subseteq V(H)$, a \emph{monochromatic $S$-tree} in $H$ is a monochromatic tree
of $H$ containing the vertices of $S$.
For a connected graph $G$ and a given integer $k$ with $2\leq k\leq |V(G)|$,
the \emph{$k$-monochromatic index $mx_k(G)$} of $G$ is the maximum number of
colors needed
such that for each subset $S\subseteq V(G)$ of $k$ vertices,
there exists a monochromatic $S$-tree.
In this paper, we prove that for any connected graph $G$, $mx_k(G)=|E(G)|-|V(G)|+2$
for each $k$ such that $3\leq k\leq |V(G)|$.

A tree $T$ in a vertex-colored graph $H$ is called a \emph{vertex-monochromatic tree}
if all the internal vertices of $T$ have the same color.
For $S\subseteq V(H)$,
a \emph{vertex-monochromatic $S$-tree} in $H$ is a vertex-monochromatic
tree of $H$ containing the vertices of $S$.
For a connected graph $G$ and a given integer $k$ with $2\leq k\leq |V(G)|$,
the \emph{$k$-monochromatic vertex-index $mvx_k(G)$} of $G$ is the maximum number of
colors needed
such that for each subset $S\subseteq V(G)$ of $k$ vertices,
there exists a vertex-monochromatic $S$-tree.
We show that for a given a connected graph $G$, and a positive integer $L$ with $L\leq |V(G)|$,
to decide whether $mvx_k(G)\geq L$ is NP-complete for each integer $k$ such that $2\leq k\leq |V(G)|$.
We also obtain some Nordhaus-Gaddum-type results for the $k$-monochromatic vertex-index.



\end{abstract}

\section{Introduction}

All graphs considered in this paper are simple, finite, undirected and connected.
We follow the terminology and notation of Bondy and Murty \cite{Bondy}.
A path in an edge-colored graph $H$ is a \emph{monochromatic path}
if all the edges of the path are colored with the same color.
The graph $H$ is called \emph{monochromatically connected}£¬
if for any two vertices of $H$ there exists a monochromatic path connecting them.
An edge-coloring of $H$ is a \emph{monochromatically connecting coloring}
(\emph{MC-coloring}) if it makes $H$ monochromatically connected.
How colorful can an MC-coloring be?
This question is the natural opposite of the well-studied problem of
rainbow connecting coloring \cite{Caro1,Chartrand1,Krivelevich and Yuster,Sun,LiSun2},
where in the latter we seek to find an edge-coloring with minimum number of colors
so that there is a rainbow path joining any two vertices.
For a connected graph $G$, the \emph{monochromatic connection number} of $G$, denoted by $mc(G)$,
is the maximum number of colors that are needed in order to make $G$
monochromatically connected.
An \emph{extremal MC-coloring} is an MC-coloring that uses $mc(G)$ colors.
These above concepts were introduced by Caro and Yuster in \cite{Caro}.
They obtained some nontrivial lower and upper bounds for $mc(G)$.
Later, Cai et al. in \cite{Di} obtained two kinds of Erd\H{o}s-Gallai-type results for $mc(G)$.

In this paper, we generalizes the concept of a monochromatic path to a monochromatic tree. In this way, we can
give the monochromatic connection number a natural generalization.
A tree $T$ in an edge-colored graph $H$ is called a \emph{monochromatic tree} if all the edges of $T$ have the same color. For an $S\subseteq V(H)$,
a \emph{monochromatic $S$-tree} in $H$ is a monochromatic tree of $H$ containing the vertices of $S$.
Given an integer $k$ with $2\leq k\leq |V(H)|$, the graph $H$ is called \emph{$k$-monochromatically connected}
if for any set $S$ of $k$ vertices of $H$, there exists a monochromatic $S$-tree in $H$.
For a connected graph $G$ and a given integer $k$ such that $2\leq k\leq |V(G)|$,
the \emph{$k$-monochromatic index $mx_k(G)$} of $G$ is the maximum number of
colors that are needed in order to make $G$ $k$-monochromatically connected.
An edge-coloring of $G$ is called a \emph{$k$-monochromatically connecting coloring} (\emph{$MX_k$-coloring}) if it makes $G$ $k$-monochromatically connected.
An \emph{extremal $MX_k$-coloring} is an $MX_k$-coloring that uses $mx_k(G)$ colors.
When $k=2$, we have $mx_2(G)=mc(G)$. Obviously, we have $mx_{|V(G)|}(G)\leq \ldots\leq mx_3(G)\leq mc(G)$.

There is a vertex version of the monochromatic connection number, which was introduced
by Cai et al. in \cite{Cai}. A path in a vertex-colored graph $H$ is a \emph{vertex-monochromatic path}
if its internal vertices are colored with the same color.
The graph $H$ is called \emph{monochromatically vertex-connected},
if for any two vertices of $H$ there exists a vertex-monochromatic path connecting them.
For a connected graph $G$, the \emph{monochromatic vertex-connection number} of $G$, denoted by $mvc(G)$,
is the maximum number of colors that are needed in order to make $G$
monochromatically vertex-connected.
A vertex-coloring of $G$ is a \emph{monochromatically vertex-connecting coloring}
(\emph{MVC-coloring}) if it makes $G$ monochromatically vertex-connected.
An \emph{extremal MVC-coloring} is an MVC-coloring that uses $mvc(G)$ colors.
This $k$-monochromatic index can also have a natural vertex version.
A tree $T$ in a vertex-colored graph $H$ is called a \emph{vertex-monochromatic tree} if its internal vertices have the same color.
For an $S\subseteq V(H)$, a \emph{vertex-monochromatic $S$-tree} in $H$ is a vertex-monochromatic tree of $H$ containing the vertices of $S$.
Given an integer $k$ with $2\leq k\leq |V(H)|$, the graph $H$ is called \emph{$k$-monochromatically vertex-connected} if for any set $S$ of $k$ vertices of $H$, there
exists a vertex-monochromatic $S$-tree in $H$.
For a connected graph $G$ and a given integer $k$ such that $2\leq k\leq |V(G)|$,
the {\it $k$-monochromatic vertex-index} $mvx_k(G)$ of $G$ is the maximum number of
colors that are needed in order to make $G$ $k$-monochromatically vertex-connected.
A vertex-coloring of $G$ is called a
\emph{$k$-monochromatically vertex-connecting coloring} (\emph{$MVX_k$-coloring}) if it makes $G$ $k$-monochromatically vertex-connected.
An \emph{extremal $MVX_k$-coloring} is an $MVX_k$-coloring that uses $mvx_k(G)$ colors.
When $k=2$, we have $mvx_2(G)=mvc(G)$.
Obviously, we have $mvx_{|V(G)|}(G)\leq \ldots\leq mvx_3(G)\leq mvc(G)$.

A {\it Nordhaus-Gaddum-type result} is a (tight) lower or upper bound on the sum or
product of the values of a parameter for a graph and its complement.
The Nordhaus-Gaddum-type is given because Nordhaus and Gaddum \cite{Nordhaus}
first established the following inequalities for the chromatic numbers of graphs:
If $G$ and $\overline{G}$ are complementary graphs on $n$ vertices
whose chromatic numbers are $\chi(G)$ and $\chi(\overline{G})$, respectively,
then $2\sqrt{n}\leq \chi{(G)}+\chi{(\overline{G})}\leq n+1$.
Since then, many analogous inequalities of other graph parameters are concerned, such
as domination number \cite{Harary}, Wiener index and some other chemical indices \cite{B. Wu},
rainbow connection number \cite{X. Li}, and so on.

In this paper, we will prove that for any connected graph $G$, $mx_k(G)=|E(G)|-|V(G)|+2$
for each $k$ such that $3\leq k\leq |V(G)|$. For the vertex version parameter $mvx_k(G)$,
we first show that for a given a connected graph $G$, and a positive integer $L$ with $L\leq |V(G)|$,
to decide whether $mvx_k(G)\geq L$ is NP-complete for each integer $k$ such that $2\leq k\leq |V(G)|$.
Then, we obtain some Nordhaus-Gaddum-type results.

\section{Determining $mx_k(G)$}

Let $G$ be a connected graph with $n$ vertices and $m$ edges.
In this section, we mainly study $mx_k(G)$ for each $k$ with $3\leq k\leq n$.
A straightforward lower bound for $mx_k(G)$ is $m-n+2$.
Just give the edges of a spanning tree of $G$ with one color,
and give each of the remaining edges a distinct new color.
A property of an extremal $MX_k$-coloring is that the edges with
each color forms a tree for any $k$ with $3\leq k\leq n$.
In fact, if an $MX_k$-coloring contains a monochromatic cycle,
we can choose any edge of this cycle and give it a new color
while still maintaining an $MX_k$-coloring;
if the subgraph induced by the edges with a given color is disconnected,
then we can give the edges of one component with a new color
while still maintaining an $MX_k$-coloring for each $k$ with $3\leq k\leq n$.
Then, we use \emph{color tree $T_c$} to denote the the tree consisting of the edges colored with $c$.
The color $c$ is called \emph{nontrivial} if $T_c$ has at least two edges; otherwise $c$ is called \emph{trivial}.
We now introduce the definition of \emph{a simple extremal $MX_k$-coloring},
which is generalized of \emph{a simple extremal MC-coloring} defined in \cite{Caro}.

Call an extremal $MX_k$-coloring \emph{simple} for a $k$ with $3\leq k\leq n$,
if for any two nontrivial colors $c$ and $d$,
the corresponding $T_c$ and $T_d$ intersect in at most one vertex.
The following lemma shows that a simple extremal $MX_k$-coloring always exists.

\begin{lem}\label{lem simple}
Every connected graph $G$ on $n$ vertices has a simple extremal $MX_k$-coloring for each $k$ with $3\leq k\leq n$.
\end{lem}
\begin{proof}
Let $f$ be an extremal $MX_k$-coloring with the most number of trivial colors for each $k$ with $3\leq k\leq n$.
Suppose $f$ is not simple. By contradiction,
assume that $c$ and $d$ are two nontrivial colors such that $T_c$ and $T_d$ contain $p$ common vertices with  $p\geq 2$. Let $H=T_c\cup T_d$. Then, $H$ is connected.
Moreover, $|V(H)|=|V(T_c)|+|V(T_d)|-p$, and $|E(H)|=|V(T_c)|+|V(T_d)|-2$.
Now color a spanning tree of $H$ with $c$,
and give each of the remaining $p-1$ edges of $H$ distinct new colors.
The new coloring is also an $MX_k$-coloring for each $k$ with $3\leq k\leq n$.
If $p>2$, then the new coloring uses more colors than $f$, contradicting that $f$ is extremal.
If $p=2$, then the new coloring uses the same number of colors as $f$ but more trivial colors,
contracting that $f$ contains the most number of trivial colors.
\end{proof}

By using this lemma, we can completely determine $mx_k(G)$ for each $k$ with $3\leq k\leq n$.

\begin{thm}
Let $G$ be a connected graph with $n$ vertices and $m$ edges,
then $mx_k(G)=m-n+2$ for each $k$ with $3\leq k\leq n$.
\end{thm}
\begin{proof}
Let $f$ be a simple extremal $MX_3$-coloring of $G$.
Choose a set $S$ of $3$ vertices of $G$. Then, there exists a monochromatic $S$-tree in $G$.
Since $|S|=3$, then this monochromatic $S$-tree is contained in some nontrivial color tree $T_c$.
Suppose that the color tree $T_c$ is not a spanning tree of $G$.
Choose $v\notin V(T_c)$, and $\{u,w\}\subseteq V(T_c)$.
Let $S'=\{v,u,w\}$. Then, there exists a monochromatic $S'$-tree in $G$.
Since $|S'|=3$, then this monochromatic $S'$-tree is contained in some nontrivial color tree $T_{d}$.
Moreover, since $v\notin V(T_c)$, then $c\neq d$.
But now, $\{u,w\}\in V(T_c)\cap V(T_{d})$, contracting that $f$ is simple.
Then, we have that $T_c$ is a spanning tree of $G$.
Hence, $m-n+2\leq mx_n(G)\leq\ldots\leq mx_3(G)\leq m-n+2$.
The theorem thus follows.
\end{proof}

\section{Hardness results for computing $mvx_k(G)$}

Through we can completely determine the value of $mx_k(G)$ for each $k$ with
$3\leq k\leq n$, for the vertex version it is difficult to compute $mvx_k(G)$ for any $k$ with
$2\leq k\leq n$. In this section, we will show that
given a connected graph $G=(V,E)$, and a positive integer $L$ with $L\leq |V|$,
to decide whether $mvx_k(G)\geq L$ is NP-complete for each $k$ with $2\leq k\leq |V|$.

We first introduce some definitions. A subset $D\subseteq V(G)$ is a \emph{dominating set}
of $G$ if every vertex not in $D$ has a neighbor in $D$.
If the subgraph induced by $D$ is connected, then $D$ is called a \emph{connected dominating set}.
The \emph{dominating number $\gamma(G)$}, and the \emph{connected dominating number $\gamma_c(G)$},
is the cardinalities of a minimum dominating set, and a minimum connected dominating set, respectively.
A graph $G$ has a connected dominating set if and only if $G$ is connected.
The problem of computing $\gamma_c(G)$ is equivalent to the problem of finding a spanning tree with the most number of leaves,
because a vertex subset is a connected dominating set if and only if its complement is contained in the set of leaves of a spanning tree.
Let $G$ be a connected graph on $n$ vertices where $n\geq 3$.
Note that the problem of computing $mvx_n(G)$ is also equivalent to the problem of finding a spanning tree with the most number of leaves.
In fact, let $T_{max}$ be a spanning tree of $G$ with the most number of leaves, and $l(T_{max})$ be the number of leaves in $T_{max}$.
Then, $mvx_n(G)=l(T_{max})+1=n-\gamma_c(G)+1$ for $n\geq 3$.
For convenience, suppose that all the graphs in this section have at least $3$ vertices.

Now we introduce a useful lemma.
For convenience, call a tree $T$ \emph{with vertex-color $c$} if the internal vertices of $T$ are colored with $c$.
\begin{lem}\label{lem MVCut}
Let $G$ be a connected graph on $n$ vertices with a cut-vertex $v_0$.
Then, $mvc(G)=l(T_0)+1$, where $T_0$ is a spanning tree of $G$ with the most number of leaves.
\end{lem}

\begin{proof}
Let $f$ be an extremal $MVC$-coloring of $G$.
Suppose that $f(v)$ is the color of the vertex $v$, and $f(v_0)=c$.
Let $G_1,G_2,\ldots ,G_p$
be the components of $G-v_0$ where $p\geq 2$.
We construct a spanning
tree $T_0$ of $G$ with vertex-color $c$ as follows.
At first, choose any pair $(v_i,v_j)\in (V(G_i),V(G_j))(i\neq j)$.
Since $v_0$ is a cut-vertex, then there must exist a
$\{v_i,v_j\}$-path $P$
containing $v_0$ with vertex-color $c$.
Initially, set $T_0=P$.
Secondly, choose another pair $(v_s,v_t)\in (V(G_s),V(G_t))(s\neq t)$
such that $v_s$ is not in $T_0$.
Similarly, there must exist a
$\{v_s,v_t\}$-path $P'$ containing $v_0$ with vertex-color $c$.
Let $x$ be the first vertex of $P'$ that is also in $T_0$,
and $y$ be the last vertex of $P'$ that is also in $T_0$.
Then, reset $T_0=T_0\cup v_sP'x\cup yP'v_t$.
Thus, $T_0$ is still a tree with vertex-color $c$ now.
Repeat the above process until all vertices are contained in $T_0$.
Finally, we get a spanning tree $T_0$ of $G$ with vertex-color $c$.
Thus, we have $mvc(G)\leq l(T_0)+1$ now.
However, $mvc(G)\geq mvx_n(G)=l(T_{max})+1$, where $T_{max}$ is a spanning
tree of $G$ with the most number of leaves.
Then, we have $l(T_0)=l(T_{max})$.
Hence, it follows that $mvc(G)=l(T_0)+1$.
\end{proof}

\begin{cor}\label{cor MVCut}
Let $G$ be a connected graph on $n$ vertices with a cut-vertex.
Then, $mvx_k(G)=l(T_{maz})+1$ for each $k$ with $2\leq k\leq n$,
where $T_{max}$ is a spanning tree of $G$ with the most number of leaves.
\end{cor}

Now, we show that the following Problem $0$ is NP-complete.

\noindent{\bf Problem $0$:} $k$-monochromatic vertex-index

\noindent Instance: Connected graph $G=(V,E)$, a positive integer $L$ with $L\leq |V|$.

\noindent Question: Deciding whether $mvx_k(G)\geq L$ for each $k$ with $2\leq k\leq |V|$.

In order to prove the NP-completeness of Problem $0$,
we first introduce the following problems.

\noindent{\bf Problem $1$:} Dominating Set.

\noindent Instance: Graph $G=(V,E)$, a positive integer $K\leq |V|$.

\noindent Question: Deciding wether there is a dominating set of size $K$ or less.

\noindent{\bf Problem $2$:} CDS of a connected graph containing a cut-vertex.

\noindent Instance: Connected graph $G=(V,E)$ with a cut-vertex, a positive integer $K$ with $K\leq |V|$.

\noindent Question: Deciding wether there is a connected dominating set of size $K$ or less.

The NP-completeness of Problem $1$ is a known result in \cite{Garey}.
In the following, we will reduce Problem $1$ to Problem $2$ polynomially.

\begin{lem}\label{lem CDS}
Problem $1$ $\preceq$ Problem $2$.
\end{lem}
\begin{proof}
Given a graph $G$ with vertex set $V=\{v_1,v_2,\ldots,v_n\}$ and edge set $E$,
we construct a graph $G'=(V',E')$ as follows:
\begin{align*}
V'= &V\cup \{u_1,u_2,\ldots,u_n\}\cup\{x,y\}\\
E'= & E\cup E_1\cup E_2\\
E_1= & \{u_iv: \text{if $v=v_i$ or $v_iv$ is an edge in $G$ for $1\leq i\leq n$}\}\\
E_2= &\{xu_i: 1\leq i\leq n\}\cup \{xy\}
\end{align*}
It is easy to check that $G'$ is connected with a cut-vertex $x$.
In the following, we will show that $G$ contains a dominating set of size $K$ or less
if and only if $G'$ contains a connected dominating set of size $K+1$ or less.
On one hand, suppose w.l.o.g that $G$ contains a dominating set
$D=\{v_1,v_2,\ldots,v_t\},t\leq K$.
Let $D'=\{u_1,u_2,\ldots,u_t\}\cup\{x\}$.
Then, it is easy to check that $D'$ is a connected dominating set of $G'$ and $|D'|\leq K+1$.
On the other hand, suppose that $G'$ contains a connected dominating set $D'$ of size $K+1$ or less.
Since $x$ is a cut-vertex of $G'$, then $x\in D'$.
For $1\leq i\leq n$, if $u_i\in D'$ or $v_i\in D'$, then put $v_i$ in $D$.
It is easy to check that $D$ is a dominating set of $G$ and $|D|\leq K$.
\end{proof}

\begin{thm}
 Problem $0$ is NP-complete.
\end{thm}
\begin{proof}
Given a connected graph $G=(V,E)$ with a cut-vertex,
and a positive integer $K$ with $K\leq |V|$.
Since $\gamma_c(G)\leq K$ if and only if $mvx_k(G)=l(T_{max})+1=|V|-\gamma_c(G)+1\geq |V|-K+1$
for $2\leq k\leq |V|$, where $T_{max}$ is a spanning tree of $G$ with the most leaves by
Corollary \ref{cor MVCut}.
Then, given a connected graph $G=(V,E)$ with a cut-vertex,
and a positive integer $L$ with $L\leq |V|$,
to decide whether $mvx_k(G)\geq L$ is NP-complete for each $k$ with $2\leq k\leq |V|$ by Lemma \ref{lem CDS}.
Moreover, Problem $0$ is NP-complete.
\end{proof}

\begin{cor}
Let $G$ be a connected graph on $n$ vertices.
Then, computing $mvx_k(G)$ is NP-hard for each $k$ with $2\leq k\leq n$.
\end{cor}

\section{Nordhaus-Gaddum-type results for $mvx_k$}

Suppose that both $G$ and $\overline{G}$ are connected graphs on $n$ vertices.
Now for $n=4$, we have $G=\overline{G}=P_4$.
It is easy to check that $mvx_k(P_4)+mvx_k(\overline{P_4})=6$ for each $k$ with $2\leq k\leq 4$.
For $k=2$, Cai et al. \cite{Cai} proved that
for $n\geq 5$, $n+3\leq mvc(G)+mvc(\overline{G})\leq 2n$,
and the bounds are sharp.
Then, in the following we suppose that $n\geq 5$ and $3\leq k\leq n$.

We first consider the lower bound of $mvx_k(G)+mvx_k(\overline{G})$ for each $k$ with $3\leq k\leq n$.
Now we introduce some useful lemmas.
\begin{lem}\cite{Peters}\label{Spanning}
If both $G$ and $\overline{G}$ are connected graphs on $n$ vertices,
then $\gamma_c(G)+\gamma_c(\overline{G})=n+1$ if and only if $G$ is the cycle $C_5$.
Moreover, if $G$ is not $C_5$, then $\gamma_c(G)+\gamma_c(\overline{G})\leq n$
with equality if and only if $\{G, \overline{G}\} = \{C_n, \overline{C_n}\}$ for $n\geq 6$,
or $\{G, \overline{G}\} = \{P_n, \overline{P_n}\}$ for $n\geq 4$,
or $\{G, \overline{G}\} = \{F_1, \overline{F_1}\}$, where $F_1$ is the graph represented in  Fig.\ref{fig1}.
\end{lem}

\begin{figure}
\centering
\includegraphics{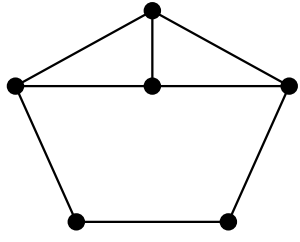}
\caption{The graph $F_1$ with $\gamma_c(F_1)=\gamma_c(\overline{F_1})=3$.}\label{fig1}
\end{figure}

\begin{lem}\cite{Cai}\label{lem MVCn}
Let $C_n$ be a cycle on $n$ vertices. Then,
\begin{align*}
mvc(C_n)=
\begin{cases}
n & n\leq 5\\
3 & n\geq 6.
\end{cases}
\end{align*}
\end{lem}

Recall that
a vertex-monochromatic $S$-tree is a vertex-monochromatic tree containing $S$.
For convenience, if the vertex-monochromatic $S$-tree is a star (with the center $v$),
we use $S$-star ($S_v$-star) to denote this vertex-monochromatic $S$-tree.
For two subsets $U,W\subseteq V(G)$,
we use $U\sim W$ to denote that any vertex in $U$ is adjacent with any vertex in $W$.
If $U=\{x\}$, we use $x\sim W$ instead of $\{x\}\sim W$.

From Lemma \ref{Spanning},
we have $mvx_k(C_n)+mvx_k(\overline{C_n})\geq mvx_n(C_n)+mvx_n(\overline{C_n})=
2n-(\gamma_c(C_n)+\gamma_c(\overline{C_n}))+2\geq n+2$ for $n\geq 6$ and $k$ with $3\leq k\leq n$.
It is easy to check that $mvx_k(C_n)=3$ for $n\geq 6$ and $k$ with $3\leq k\leq n$ by Lemma \ref{lem MVCn}.
Then, we have $mvx_k(\overline{C_n})\geq n-1$ for $n\geq 6$ and $k$ with $3\leq k\leq n$.
Now we introduce the following lemma.

\begin{lem}\label{lem Cn}
For $n\geq 6$, if $n$ is odd, then $mvx_k(\overline{C_n})=n$ for $k$ with $3\leq k\leq \frac{n-1}{2}$,
and $mvx_k(\overline{C_n})=n-1$ for $k$ with $\frac{n+1}{2}\leq k\leq n$;
if $n=4t$, then $mvx_k(\overline{C_n})=n$ for $k$ with $3\leq k\leq \frac{n}{2}-1$,
and $mvx_k(\overline{C_n})=n-1$ for $k$ with $\frac{n}{2}\leq k\leq n$;
if $n=4t+2$, then $mvx_k(\overline{C_n})=n$ for $k$ with $3\leq k\leq \frac{n}{2}$,
and $mvx_k(\overline{C_n})=n-1$ for $k$ with $\frac{n}{2}+1\leq k\leq n$.
\end{lem}
\begin{proof}
Suppose that $V(C_n)=\{v_0,v_1,\ldots,v_{n-1}\}$,
and the clockwise permutation sequence is $v_0,v_1,\ldots,v_{n-1},v_0$ in $C_n$.
Let $f$ be an extremal $MVX_k$-coloring of $\overline{C_n}$ for each $k$ with $3\leq k\leq n$.
Suppose first that $n$ is odd.
Let $S=\{v_i:i\equiv 0 ~or~ 1\pmod {4}$\}. Then, $|S|=\frac{n+1}{2}$.
It is easy to check that there exists no $S$-star in $\overline{C_n}$.
Then, we have $mvx_k(\overline{C_n})< n$ for $k$ with $\frac{n+1}{2}\leq k\leq n$.
Hence, $mvx_k(\overline{C_n})=n-1$ for $k$ with $\frac{n+1}{2}\leq k\leq n$.
For $k$ with $3\leq k\leq \frac{n-1}{2}$, we will show that $mvx_k(\overline{C_n})=n$.
In other words, for any set $S$ of $k$ vertices of $\overline{C_n}$, there
exists an $S$-star in $\overline{C_n}$.
We first show that $mvx_k(\overline{C_n})$ for $k=\frac{n-1}{2}$.
By contradiction, assume that $mvx_k(\overline{C_n})<n$ for $k=\frac{n-1}{2}$.
Suppose that $S$ is a set of $k$ vertices such that there
exists no $S$-star in $\overline{C_n}$.
Note that the vertex-induced subgraph $C_n[S]$ consists of some disjoint paths
$\{P_{v_{i_1}v_{j_1}},P_{v_{i_2}v_{j_2}},\ldots,P_{v_{i_p}v_{j_p}}\}$
where $\{v_{i_q},v_{j_q}\}$ denote the ends of $P_{v_{i_q}v_{j_q}}$
such that the vertex-sequence $v_{i_q}$ to $v_{j_q}$ along $P_{v_{i_q}v_{j_q}}$
is in clockwise direction in $C_n$ for each $q$ with $1\leq q\leq p$.

{\bf Claim $1$:} Each $P_{v_{i_q}v_{j_q}}$ contains at least $2$ vertices for each $q$ with $1\leq q\leq p$.

{\bf Proof of Claim $1$:} By contradiction, assume that $P_{v_{i_q}v_{j_q}}=v$ for some $v\in V(C_n)$ now.
Since $\{P_{v_{i_1}v_{j_1}},P_{v_{i_2}v_{j_2}},\ldots,P_{v_{i_p}v_{j_p}}\}$ are disjoint paths in $C_n$,
then $v\sim S\setminus\{v\}$ in $\overline{C_n}$.
Hence, there exists an $S_v$-star in $\overline{C_n}$, a contradiction.

Consider $\{P_{v_{i_1}v_{j_1}},P_{v_{i_2}v_{j_2}},\ldots,P_{v_{i_p}v_{j_p}}\}$ in $C_n$.
Suppose w.l.o.g that the clockwise permutation sequence of these paths is
$P_{v_{i_1}v_{j_1}},P_{v_{i_2}v_{j_2}},\ldots,P_{v_{i_p}v_{j_p}},P_{v_{i_{p+1}}v_{j_{p+1}}}=P_{v_{i_1}v_{j_1}}$ in $C_n$.
For any two successive paths $P_{v_{i_q}v_{j_q}}$ and $P_{v_{i_{q+1}}v_{j_{q+1}}}$ where $1\leq q\leq p$,
we have the following claim.

{\bf Claim $2$:} There are at most $2$ vertices between $\{v_{j_q},v_{i_{q+1}}\}$
in clockwise direction in $C_n$ for each $q$ with $1\leq q\leq p$.

{\bf Proof of Claim $2$:} By contradiction, assume that there are at least $3$ vertices $\{v_{r-1},v_{r},v_{r+1}\}$, where the subscript is subject to modulo $n$,
between $\{v_{j_q},v_{i_{q+1}}\}$ in clockwise direction in $C_n$.
Now, we have $v_r\sim S$ in $\overline{C_n}$.
Then, there exists an $S_{v_r}$-star in $\overline{C_n}$, a contradiction.

If $n=4t+1$, then $k=2t$. Now, we have $p\leq \lfloor\frac{k}{2}\rfloor=t$ by Claim $1$.
Then, $|V(C_n)|\leq k+2p\leq n-1<n$ by Claim $2$, a contradiction.
If $n=4t+3$, then $k=2t+1$. Now, we have $p\leq \lfloor\frac{k}{2}\rfloor=t$ by Claim $1$.
Then, $|V(C_n)|\leq k+2p\leq n-2<n$ by Claim $2$, a contradiction.
Hence, if $n$ is odd, then $n=mvx_{\frac{n-1}{2}}(\overline{C_n})\leq \ldots mvx_4(\overline{C_n})\leq mvx_3(\overline{C_n})\leq n$.
The proof for the case $n=4t$ or $n=4t+2$ is similar. We omit their details.
\end{proof}

\begin{thm}
Suppose that both $G$ and $\overline{G}$ are connected graphs on $n$ vertices.
For $n=5$, $mvx_k(G)+mvx_k(\overline{G})\geq 6$ for $k$ with $3\leq k\leq 5$.
For $n=6$, $mvx_k(G)+mvx_k(\overline{G})\geq 8$ for $k$ with $3\leq k\leq 6$.
For $n\geq 7$,
if $n$ is odd, then $mvx_k(G)+mvx_k(\overline{G})\geq n+3$ for $k$ with $3\leq k\leq \frac{n-1}{2}$,
and $mvx_k(G)+mvx_k(\overline{G})\geq n+2$ for $k$ with $\frac{n+1}{2}\leq k\leq n$;
if $n=4t$, then $mvx_k(G)+mvx_k(\overline{G})\geq n+3$ for $k$ with $3\leq k\leq \frac{n}{2}-1$,
and $mvx_k(G)+mvx_k(\overline{G})\geq n+2$ for $k$ with $\frac{n}{2}\leq k\leq n$;
if $n=4t+2$, then $mvx_k(G)+mvx_k(\overline{G})\geq n+3$ for $k$ with $3\leq k\leq \frac{n}{2}$,
and $mvx_k(G)+mvx_k(\overline{G})\geq n+2$ for $k$ with $\frac{n}{2}+1\leq k\leq n$.
Moreover, all the above bounds are sharp.
\end{thm}
\begin{proof}
For $n=5$, if $G=\overline{G}=C_5$,
then it is easy to check that $2mvx_k(C_5)=6$ for $k$ with $3\leq k\leq 5$;
if $G\neq C_5$, then $mvx_k(G)+mvx_k(\overline{G})\geq 7$ for $k$ with $3\leq k\leq 5$ by Lemma \ref{Spanning}.
For $n\geq 6$, we have $mvx_k(G)+mvx_k(\overline{G})\geq mvx_n(G)+mvx_n(\overline{G})=n+2$ for $k$ with $3\leq k\leq n$
with equality if and only if $\{G, \overline{G}\} = \{C_n, \overline{C_n}\}$ for $n\geq 6$,
or $\{G, \overline{G}\} = \{P_n, \overline{P_n}\}$ for $n\geq 6$,
or $\{G, \overline{G}\} = \{F_1, \overline{F_1}\}$, where $F_1$ is the graph represented in  Fig.\ref{fig1} by Lemma \ref{Spanning}.
For $n\geq 6$,
it is easy to check that $mvx_k(C_n)=mvx_k(P_n)=3$ for $k$ with $3\leq k\leq n$ by Lemma \ref{lem MVCn}.
Then, we have $mvx_k(P_n)+mvx_k(\overline{P_n})\geq mvx_k(C_n)+mvx_k(\overline{C_n})$ for $k$ with $3\leq k\leq n$.
Furthermore, for $n=6$, it is easy to check that $mvx_k(F_1)+mvx_k(\overline{F_1})=8$ for $k$ with $3\leq k\leq 6$. Thus, the theorem follows for $n\geq 6$ by Lemma \ref{lem Cn}.
\end{proof}

Now we consider the upper bound of $mvx_k(G)+mvx_k(\overline{G})$ for each $k$ with $\lceil\frac{n}{2}\rceil\leq k\leq n$.
For convenience, we use $d_G(v)$ and $N_G(v)$ to denote the degree and the neighborhood of a vertex $v$ in $G$.
For any two vertices $u,v \subseteq V(G)$, we use $d_G(u,v)$ to denote the distance between $u$ and $v$ in $G$.
Note that a straightforward upper bound of $mvx_k(G)$ is that $mvx_k(G)\leq mvc(G)\leq n-diam(G)+2$
where $diam(G)$ is the diameter of $G$ for each $k$ with $3\leq k\leq n$.
Next we introduce some useful lemmas.

\begin{lem}\label{Kmn}
Let $K_{n_1,n_2}$ be a complete bipartite graph such that $n=n_1+n_2$, and $n_1,n_2\geq 2$.
Let $G=K_{n_1,n_2}-e$, where $e$ is an edge of $K_{n_1,n_2}$.
Then, $mvx_k(G)+mvx_k(\overline{G})=2n-2$ for $3\leq k\leq n$.
\end{lem}
\begin{proof}
It is easy to check that $diam(G)=3$, and $diam(\overline{G})=3$.
Then, we have $mvc(G)+mvc(\overline{G})\leq 2n-2$.
It is also easy to check that both $G$ and $\overline{G}$ contain a double star as a spanning tree.
Then, we have $mvx_n(G)+mvx_n(\overline{G})\geq 2n-2$.
Hence, the lemma follows by the fact that $mvx_n(G)\leq\ldots\leq mvx_3(G)\leq mvc(G)$.
\end{proof}

\begin{lem}\label{Upp}
If $k=\lceil\frac{n}{2}\rceil$,
then $mvx_k(G)+mvx_k(\overline{G})\leq 2n-2$ for $n\geq 5$.
\end{lem}
\begin{proof}
Let $V(G)=\{v_1,v_2,\ldots,v_n\}$.
Since $\overline{G}$ is connected, then $\Delta(G)\leq n-2$.
Suppose first that $mvx_k=n$, and $f$ is an extremal $MVX_k$-coloring of $G$.
Then, for any set $S$ of $k$ vertices of $G$, there
exists an $S$-star in $G$.
This also implies that $\Delta(G)\geq k-1$.

{\bf Case 1:} $\Delta(G)\geq n-k+1$.

Suppose w.l.o.g that $d_G(v_1)=\Delta(G)$, and $N_G(v_1)=\{v_2,v_3,\ldots,v_{\Delta+1}\}$.
Let $S=\{v_1,v_{\Delta+2},\ldots,v_{n-1},v_{n}\}$.
Since $|S|=n-\Delta(G)\leq k-1<k$,
then there exists an $S_v$-star in $G$.
Moreover, since $v_1\nsim \{v_{\Delta+2},\ldots,v_{n-1},v_{n}\}$ in $G$, then $v\in N_G(v_1)$.
Suppose w.l.o.g that $v=v_2$.
Then, we have $d_{\overline{G}}(v_1,v_2)\geq 3$.
Since $d_{\overline{G}}(v_1,v_2)\geq 3$,
then $mvx_k(\overline{G})\leq n-diam(\overline{G})+2\leq n-1$.
Suppose $mvx_k(\overline{G})=n-1$. Then, $diam(\overline{G})=3$.
Let $g$ be an extremal $MVX_k$-coloring of $\overline{G}$.
Note that if $\overline{G}$ is $k$-monochromatically vertex-connected,
it is also monochromatically vertex-connected.
Since $mvx_k(\overline{G})=n-1$,
then there exists a vertex-monochromatic path $P=v_1xyv_2$ of length $3$ in $\overline{G}$
such that $x\in \{v_{\Delta+2},\ldots,v_{n-1},v_{n}\}$, and $y\in N_G(v_1)\setminus \{v_2\}$.
Suppose w.l.o.g that $P=v_1v_{\Delta+2}v_{\Delta+1}v_2$.
This also implies that $v_{\Delta+1}\nsim \{v_2,v_{\Delta+2}\}$ in $G$.
Let $S'=\{v_1,v_{\Delta+1},v_{\Delta+2},\ldots,v_n\}$ now.
Since $|S'|=n-\Delta(G)+1\leq k$, then there exists an $S_{v'}'$-star in $G$.
Moreover, since $v_1\nsim \{v_{\Delta+2},\ldots,v_{n-1},v_{n}\}$ and $v_{\Delta+1}\nsim \{v_2,v_{\Delta+2}\}$ in $G$, then $v'\in N_G(v_1)\setminus \{v_2,v_{\Delta+1}\}$.
Now, we have $d_{\overline{G}}(v_1,v')= 3$.
Since $mvx_k(\overline{G})=n-1$,
then $\{v_{\Delta+1},v_{\Delta+2}\}$ are the only two vertices with the same color in $\overline{G}$.
But now, since $v'\nsim\{v_{\Delta+1},v_{\Delta+2}\}$ in $\overline{G}$,
then there exists no vertex-monochromatic path connecting $\{v_1,v'\}$ in $\overline{G}$,
a contradiction.
Hence, we have that $mvx_k(\overline{G})\leq n-2$, and $mvx_k(G)+mvx_k(\overline{G})\leq 2n-2$.

{\bf Case 2:} $\Delta(G)\leq n-k$.

Since $k=\lceil\frac{n}{2}\rceil$, and $\Delta(G)\geq k-1$,
then $\lceil\frac{n}{2}\rceil-1\leq \Delta(G)\leq n-\lceil\frac{n}{2}\rceil$.

If $n$ is odd, then $\Delta(G)=\frac{n-1}{2}=k-1$.
Suppose w.l.o.g that $d_G(v_1)=\Delta(G)$, and $N_G(v_1)=\{v_2,v_3,\ldots,v_{k}\}$.
Let $S=\{v_1,v_{k+1},\ldots,v_n\}$.
Since $|S|=n-k+1=k$, then there exists an $S_v$-star in $G$.
Moreover, since $v_1\nsim \{v_{k+1},\ldots,v_{n-1},v_{n}\}$ in $G$, then $v$ is not in $S$.
But now, $d_G(v)\geq |S|=k>\Delta(G)$, a contradiction.

If $n$ is even, then $\Delta(G)=\frac{n}{2}-1$ or $\frac{n}{2}$.
Suppose w.l.o.g that $d_G(v_1)=\Delta(G)$, and $N_G(v_1)=\{v_2,v_3,\ldots,v_{\Delta+1}\}$.
If $\Delta(G)=\frac{n}{2}-1=k-1$,
then let $S=\{v_1,v_{k+1},\ldots,v_{n-1}\}$.
Since $|S|=n-k=k$, then there exists an $S_v$-star in $G$.
Moreover, since $v_1\nsim \{v_{k+1},\ldots,v_{n-1}\}$ in $G$, then $v$ is not in $S$.
But now, $d_G(v)\geq |S|=k>\Delta(G)$, a contradiction.
If $\Delta(G)=\frac{n}{2}=k$, then let $S=\{v_1,v_{k+2},\ldots,v_n\}$.
Since $|S|=n-k=k$, then there exists an $S_v$-star in $G$.
Moreover, since $v_1\nsim \{v_{k+2},\ldots,v_{n-1},v_{n}\}$ in $G$, then $v\in N_G(v_1)$.
Suppose w.l.o.g that $v=v_2$. Then, $d_G(v_2)=k=\Delta(G)$, and $N_G(v_2)=\{v_1,v_{k+2},\ldots,v_n\}$ now.
If $k\geq 4$, then let $S'=\{v_1,v_2,v_{k+1},v_{k+2}\}$.
Since $|S'|\leq k$, then there exists an $S_{v'}'$-star in $G$.
But now, since $v_1\nsim v_{k+2}$, and $v_2\nsim v_{k+1}$ in $G$,
then $v'\in N_G(v_1)\cap N_G(v_2)=\emptyset$, a contradiction.
If $k=3$, then $n=6$ now. If $\{v_2,v_3,v_4\}\sim \{v_5,v_6\}$ in $G$,
then $G$ contains a complete bipartite spanning subgraph.
But now, $\overline{G}$ is not connected, a contradiction.
So, suppose w.l.o.g that $v_4\nsim v_5$ in $G$.
Similarly consider $S'=\{v_1,v_3,v_5\}$,$\{v_1,v_4,v_5\}$, $\{v_1,v_4,v_6\}$,
and $\{v_3,v_5,v_6\}$, respectively.
Then, we will have that $v_3\sim v_5$, $v_3\sim v_4$, $v_4\sim v_6$, and $v_5\sim v_6$ in $G$, respectively.
But now, $\overline{G}$ is contained in a cycle $C_6$. Then, $mvx_3(\overline{G})\leq mvx_3(C_6)=3$.
So, for $n=6$ we have $mvx_3(G)+mvx_3(\overline{G})\leq n+3<2n-2$.

Now suppose w.l.o.g that $mvx_k(G)\leq n-1$, and $mvx_k(\overline{G})\leq n-1$, respectively.
Thus, we also have $mvx_k(G)+mvx_k(\overline{G})\leq 2n-2$.
\end{proof}

\begin{thm}
Suppose that both $G$ and $\overline{G}$ are connected graphs on $n\geq 5$ vertices.
Then, for $k$ with $\lceil\frac{n}{2}\rceil\leq k\leq n$, we have that
$mvx_k(G)+mvx_k(\overline{G})\leq 2n-2$,
and this bound is sharp.
\end{thm}
\begin{proof}
For $k$ with $\lceil\frac{n}{2}\rceil\leq k\leq n$, we have $mvx_k(G)\leq mvx_{\lceil\frac{n}{2}\rceil}\leq 2n-2$ by Lemma \ref{Upp}.
From Lemma \ref{Kmn}, this bound is sharp for $k$ with $\lceil\frac{n}{2}\rceil\leq k\leq n$.
\end{proof}

\end{document}